\documentclass[11pt,a4paper]{article}
\usepackage{amsmath,amssymb,amsthm,mathrsfs,fullpage,bbm}
\usepackage{mathtools,microtype}
\allowdisplaybreaks
\usepackage[colorlinks=true,linkcolor=blue,citecolor=blue,urlcolor=blue]{hyperref}

\newtheorem{theorem}{Theorem}[section]
\newtheorem{lemma}[theorem]{Lemma}
\theoremstyle{remark}

\numberwithin{equation}{section}

\newcommand{\A}{\mathcal{A}}
\newcommand{\E}{\mathbb{E}}
\newcommand{\W}{\mathrm{W}}
\newcommand\tup[1]{\left\langle #1 \right\rangle}

\title{Quantitative bounds for regular $3$-wise intersecting families}
\author{Fan Chang\thanks{School of Statistics and Data Science, Nankai University, Tianjin, China; and Extremal Combinatorics and Probability Group, Institute for Basic Science, Daejeon, South Korea. Email: \texttt{1120230060@mail.nankai.edu.cn}. Supported by the National Natural Science Foundation of China (NSFC) under grant 124B2019 and by the Institute for Basic Science (IBS-R029-C4).}
}
\date{}

\begin{document}
\maketitle

\begin{abstract}
Frankston, Kahn and Narayanan proved that every regular increasing 3-wise intersecting family of subsets of $[n]$ has cardinality $o(2^n)$ using Friedgut’s junta theorem. We give a short quantitative proof using elementary tools from the analysis of Boolean functions and entropy. More precisely, if $\A\subseteq\mathcal P_n$ is a nonempty 3-wise intersecting family that is both regular and increasing, then
$$
\log\frac{2^n}{|\mathcal{A}|}\ge
\frac{n}{2}\left(\frac{|\mathcal{A}|}{2^n-|\mathcal{A}|}\right)^2,
$$
and consequently $|\mathcal{A}|\le 2^n\sqrt{\W(n)/n}$, where $\W$ is the principal Lambert function defined by $\W(x)e^{\W(x)}=x$ for $x\ge0$. We also give a purely Fourier-analytic proof of the weaker estimate 
$$
|\mathcal{A}|\le \frac{2^n}{1+n^{1/3}}.
$$
\end{abstract}

\section{Introduction}

For a positive integer $n$, let $[n]=\{1,2,\ldots,n\}$ and let $\mathcal P_n$ denote the power-set of $[n]$. For an integer $r\ge2$, a family $\A\subseteq\mathcal P_n$ is said to be \emph{$r$-wise intersecting} if any $r$ of the sets in $\A$ have nonempty intersection. We say that $\A$ is \emph{increasing} if it is closed under taking supersets, \emph{regular} if every element of $[n]$ belongs to the same number of members of $\A$, and \emph{symmetric} if its automorphism group is transitive on $[n]$.

The distinction between $2$-wise and $3$-wise intersection is already apparent in the symmetric setting. When $n$ is odd, the family $\{A\subseteq[n]:|A|>n/2\}$ is a symmetric intersecting family of size $2^{n-1}$, which is the largest possible size of an intersecting subfamily of $\mathcal{P}_n$. For $3$-wise intersection, however, Frankl~\cite{F1981} conjectured that every symmetric $3$-wise intersecting family has cardinality $o(2^n)$. Cameron, Frankl and Kantor~\cite{CFK1989} had earlier obtained a substantially stronger estimate in the $4$-wise setting. Frankl's conjecture was eventually proved by Ellis and Narayanan~\cite{EN2017}, who obtained the quantitative bound $|\A|\le 2^n/n^c$ for some universal constant $c>0$ by combining the $p$-biased measure with the Friedgut--Kalai sharp-threshold theorem~\cite{FK1996}. A construction of Riordan, recorded in~\cite{EN2017}, gives symmetric $3$-wise intersecting families satisfying
$$
\log_2|\A|=n-2\sqrt n+o(\sqrt n)
$$
for infinitely many $n$, and therefore they conjectured that every symmetric 3-wise intersecting family $\A\subseteq\mathcal P_n$ satisfies
$$
\log_2|\A|\le n-cn^\delta
$$
for some universal constants $c,\delta>0$ and one cannot take $\delta >1/2$ in such a result.

Symmetry is considerably stronger than regularity, and Frankl~\cite{F1981} gave a projective-geometric construction of regular $3$-wise intersecting families containing a positive proportion of all subsets of the ground set, so regularity alone cannot imply an $o(2^n)$ bound. Frankston, Kahn and Narayanan~\cite{FKN2018} showed that every regular increasing $3$-wise intersecting family has cardinality $o(2^n)$, using Friedgut's junta theorem~\cite{F1998} to establish the required threshold behaviour. The quantitative estimate obtained by their argument is, however, very weak, and they raised the corresponding problem for regular increasing families. Analogous questions for vector-intersecting families were considered by Eberhard, Kahn, Narayanan and Spirkl \cite{EKNS2021}, while the more recent theory of global functions has yielded effective quantitative bounds for intersecting families under weaker ``smearedness'' assumptions \cite{KLM2025JEMS}.

This short note aims to give a direct quantitative proof of the theorem of Frankston, Kahn and Narayanan~\cite{FKN2018}. Throughout the paper, $\log$ denotes the natural logarithm. Our main result is the following.

\begin{theorem}\label{thm:main}
If $\A\subseteq\mathcal P_n$ is a nonempty 3-wise intersecting family that is both regular and increasing, then
$$
\log\frac{2^n}{|\A|}\ge\frac{n}{2}\left(\frac{|\A|}{2^n-|\A|}\right)^2.
$$
In particular,
$$
|\A|\le2^n\sqrt{\frac{\W(n)}n},
$$
where $\W$ is the principal branch of the Lambert $W$-function defined by $\W(x)e^{\W(x)}=x$ for $x\ge0$.
\end{theorem}

Since $\W(n)\le\log n$ for $n\ge3$, Theorem~\ref{thm:main} gives
$$
\log_2|\A|\le n-\frac12\log_2n+\frac12\log_2\log n.
$$

Using only basic tools from the analysis of Boolean functions, we can prove the following weaker estimate.

\begin{theorem}\label{thm:fourier}
If $\A\subseteq\mathcal P_n$ is a 3-wise intersecting family that is both regular and increasing, then
$$
|\A|\le\frac{2^n}{1+n^{1/3}}.
$$
\end{theorem}

Both theorems also hold for symmetric $3$-wise intersecting families. Indeed, the upward closure of such a family is $3$-wise intersecting, regular and increasing.

Let us briefly describe the proofs. We identify $\mathcal P_n$ with $\{0,1\}^n$ and write $f=\mathbbm{1}_{\A}$ and $\alpha=\E[f]$. The combinatorial hypothesis is converted into spectral information in two steps: 3-wise intersection implies that $\A$ is sum-free in $\mathbb F_2^n$, and sum-freeness gives the cubic identity
\[
\sum_{S\subseteq[n]}\hat{f}(S)^3=0.
\]
On the other hand, regularity and monotonicity imply that all coordinate influences are equal, while every nonconstant Fourier coefficient is bounded by half of the relevant influence. Combining these facts with Parseval's identity gives
\[
{\rm I}[f]\ge \frac{2n\alpha^2}{1-\alpha} \ge 2n\alpha^2.
\]
A second use of Parseval gives ${\rm I}[f]\le2\sqrt{n\alpha(1-\alpha)}$, and Theorem~\ref{thm:fourier} follows. For Theorem~\ref{thm:main}, the same lower bound on ${\rm I}[f]$ forces a bias in every coordinate of a uniformly random member of $\A$; subadditivity of entropy then turns this common bias into the stated estimate.

\medskip\noindent\emph{Organization.} Section~\ref{sec:prelim} collects the elementary facts from the analysis of Boolean functions and entropy. Theorem~\ref{thm:fourier} is proved in Section~\ref{sec:fourier}, and Therorem~\ref{thm:main} in Section~\ref{sec:entropy}. We conclude in Section~\ref{sec:remark} with a brief discussion of the limitation of the method.

\section{Preliminaries}\label{sec:prelim}

In this section, we briefly describe the notions and tools we shall require for our arguments. We identify $\mathcal P_n$ with the discrete cube $\{0,1\}^n$, and, when convenient, with the group $\mathbb F_2^n$ under coordinatewise addition modulo $2$.

We consider real-valued functions $f:\{0,1\}^n \to \mathbb{R}$, equipped with the inner product $\tup{f,g}=\mathbb{E}_x[f(x)g(x)]$. For $S\subseteq[n]$, define the Fourier--Walsh character $\chi_S(x):=(-1)^{\sum_{i\in S}x_i}$. The family $
\{\chi_S\}_{S\subseteq[n]}$ is an orthonormal basis of $L^2(\{0,1\}^n)$. The Fourier--Walsh expansion of $f$ is given by $f(x)=\sum_{S\subseteq[n]}\hat{f}(S)\chi_S(x)$, where $\hat{f}(S)=\tup{f,\chi_S}$.
We shall use Parseval's identity $\sum_{S\subseteq[n]}\hat{f}(S)^2=\E[f^2]$. We refer the reader to~\cite{R2014} for the standard background on Boolean Fourier analysis.

For $i\in[n]$, let $e_i$ denote the $i$th standard basis vector. If $f:\{0,1\}^n\to\{0,1\}$, define the \emph{$i$th influence} and the \emph{total influence} by
$$
{\rm Inf}_i[f]=\Pr_x\bigl(f(x)\ne f(x \oplus e_i)\bigr),
\qquad
{\rm I}[f]=\sum_{i=1}^n {\rm Inf}_i[f].
$$
For $x\in\{0,1\}^{[n]\setminus\{i\}}$ and $a\in\{0,1\}$, we write $x^{i=a}$ for the point obtained by setting the $i$th coordinate equal to $a$. Then
$$
{\rm Inf}_i[f]=\E_{x\in\{0,1\}^{[n]\setminus\{i\}}}\left[
\left|f(x^{i=1})-f(x^{i=0})\right|\right].
$$
We call $f$ increasing or regular when its support has the corresponding property as a family of subsets of $[n]$.

A set $\A\subseteq\mathbb F_2^n$ is \emph{sum-free} if $x,y\in\A$ implies $x+y\notin\A$.

\begin{lemma}\label{lem:cubic}
If $\A\subseteq\mathbb F_2^n$ is sum-free and $f=\mathbbm{1}_{\A}$, then
\[
\sum_{S\subseteq[n]}\hat{f}(S)^3=0.
\]
\end{lemma}

\begin{proof}
Sum-freeness gives $f(x)f(y)f(x+y)=0$ for all $x,y\in\mathbb F_2^n$. Hence
\[
0=\E_{x,y}[f(x)f(y)f(x+y)].
\]
Expanding each $f$ in Fourier expansions and using $\chi_S(x+y)=\chi_S(x)\chi_S(y)$, we obtain
$$
0=\sum_{R,S,T\subseteq[n]}
\hat{f}(R)\hat{f}(S)\hat{f}(T)\E_x[\chi_R(x)\chi_T(x)]\E_y[\chi_S(y)\chi_T(y)]=\sum_{S\subseteq[n]}\hat{f}(S)^3,
$$
by orthogonality.
\end{proof}

\begin{lemma}\label{lem:regular-influences}
If $f:\{0,1\}^n\to\{0,1\}$ is regular and increasing, then ${\rm Inf}_i[f]=\frac{{\rm I}[f]}n$ for every $i\in[n]$.
\end{lemma}

\begin{proof}
Let $\A$ be the support of $f$, and for each $i\in[n]$ define
\[
\A_i^0=\{x\subseteq[n]\setminus\{i\}:x\in\A\},
\qquad
\A_i^1=\{x\subseteq[n]\setminus\{i\}:x\cup\{i\}\in\A\}.
\]
Since $\A$ is increasing, $\A_i^0\subseteq\A_i^1$, and therefore
\[
{\rm Inf}_i[f]=\frac{|\A_i^1\setminus\A_i^0|}{2^{n-1}}=\frac{2|\A_i^1|-|\A|}{2^{n-1}}.
\]
Regularity says that $|\A_i^1|$ is independent of $i$. Thus all the influences are equal, and summing them gives the claim.
\end{proof}

\begin{lemma}\label{lem:coefficient-influence}
Let $f:\{0,1\}^n\to\{0,1\}$ be increasing. If $i\in S$, then
$$
|\hat{f}(S)|\le \frac12 {\rm Inf}_i[f].
$$
In particular, $\hat{f}(\{i\})=
-\frac12 {\rm Inf}_i[f]$.
\end{lemma}
\begin{proof}
Write $x=(x_i,z)\in\{0,1\}^n$ with $z\in\{0,1\}^{[n]\setminus\{i\}}$. Note that
\begin{equation}
    \begin{split}
  |\hat{f}(S)|&=\left|\mathbb{E}_{x\in\{0,1\}^n}[f(x)\chi_{S\setminus\{i\}}(z)(-1)^{x_i}]\right|\\
  &=\frac{1}{2}\left|\E_{z\in\{0,1\}^{[n]\setminus\{i\}}}\left[\chi_{S\setminus\{i\}}(z)\bigl(f(z^{i=1})-f(z^{i=0})\bigr)\right]\right|\\
  &\le\frac{1}{2}\E_{z\in\{0,1\}^{[n]\setminus\{i\}}}\left[\left|f(z^{i=1})-f(z^{i=0})\right|\right]=\frac12 {\rm Inf}_i[f]
    \end{split}
\end{equation}
When $S=\{i\}$, monotonicity gives
$f(z^{i=1})-f(z^{i=0})\ge0$, and therefore $\hat{f}(\{i\})=
-\frac12 {\rm Inf}_i[f]$.
\end{proof}

We briefly recall the elementary facts about entropy that will be used below; see, for example,~\cite[Section~14.6]{AS2000}. Let $X$ be a discrete random variable with finite support $\Omega$, and write $p_X(x)=\Pr(X=x)$. The \emph{Shannon entropy} of $X$ is defined by
$$
H(X)=-\sum_{x\in\Omega}p_X(x)\log p_X(x).
$$
For random variables $X_1,\ldots,X_n$, we write $H(X_1,\ldots,X_n)$ for the entropy of the joint random variable $(X_1,\ldots,X_n)$. In particular, if $X$ is uniformly distributed on a finite set $\Omega$, then $H(X)=\log|\Omega|$.

A Bernoulli random variable with parameter $t$ has entropy $h(t)=
-t\log t-(1-t)\log(1-t)$ for $0\le t\le1$, where, as usual, $0\log0=0$. We shall use the standard subadditivity inequality
$$
H(X_1,\ldots,X_n)\le\sum_{i=1}^n H(X_i).
$$

The following immediate consequence will be used in the proof of Theorem~\ref{thm:main}.

\begin{lemma}\label{lem:entropy-regular}
Let $Z$ be uniformly distributed on a regular family $\A\subseteq\mathcal P_n$. Then there exists $\theta\in[0,1]$ such that $\Pr(i\in Z)=\theta$ for every $i\in[n]$, and
$$
\log|\A|\le nh(\theta).
$$
\end{lemma}

\begin{proof}
Identify $Z=(Z_1,\ldots,Z_n)$ with $Z_i=\mathbbm{1}_{\{i\in Z\}}$. Since $\A$ is regular, each $Z_i$ is a Bernoulli random variable with the same parameter $\theta=\Pr(i\in Z)$. Then subadditivity gives
$$
\log|\A|=H(Z_1,\ldots,Z_n)\le\sum_{i=1}^n H(Z_i)=nh(\theta),
$$
as required.
\end{proof}

\begin{lemma}\label{lem:pinsker}
For every $0\le t\le1$,
$$
\log2-h(t)\ge 2\left(t-\frac12\right)^2.
$$
\end{lemma}

\begin{proof}
Set $g(t)=\log2-h(t)-2(t-\frac12)^2$. We have $g(\frac12)=g'(\frac12)=0$, while
\[
g''(t)=\frac1{t(1-t)}-4\ge0.
\]
Thus $g$ is convex and has its minimum at $\frac12$.
\end{proof}

\section{Proof of Theorem~\ref{thm:fourier}}\label{sec:fourier}

We first give the simple combinatorial observation that allows us to use Lemma~\ref{lem:cubic}.

\begin{lemma}\label{lem:threewise-sumfree}
If $\A\subseteq\mathcal P_n$ is 3-wise intersecting, then $\A$ is sum-free in $\mathbb F_2^n$.
\end{lemma}

\begin{proof}
For any $x,y\in\mathbb F_2^n$, the three sets corresponding to $x$, $y$ and $x+y$ have empty common intersection. Indeed, if $x_i=y_i=1$, then $(x+y)_i=0$, while otherwise at least one of $x_i,y_i$ is zero. Thus $x,y,x+y$ cannot all belong to a 3-wise intersecting family.
\end{proof}

\begin{proof}[Proof of Theorem~\ref{thm:fourier}]
Let $f=\mathbbm{1}_{\A}$ and set $\alpha=\E[f]=\frac{|\A|}{2^n}$. Since $\A$ is intersecting, $0<\alpha\le1/2$.

By Lemmas~\ref{lem:threewise-sumfree} and~\ref{lem:cubic},
$$
\alpha^3=-\sum_{S\ne\varnothing}\hat f(S)^3.
$$
Also, Lemmas~\ref{lem:regular-influences} and \ref{lem:coefficient-influence} give
$$
|\hat{f}(S)|\le\frac{{\rm I}[f]}{2n}
$$
for every nonempty $S\subseteq[n]$. It follows from Parseval's identity that
$$
\alpha^3\le \sum_{S\ne\varnothing}|\hat{f}(S)|^3\le\frac{{\rm I}[f]}{2n}
\sum_{S\ne\varnothing}\hat{f}(S)^2=\frac{{\rm I}[f]}{2n}\alpha(1-\alpha).
$$
Thus
\begin{equation}\label{eq:influence-lower}
{\rm I}[f]\ge\frac{2n\alpha^2}{1-\alpha}
\ge 2n\alpha^2.
\end{equation}

On the other hand, Lemmas~\ref{lem:regular-influences} and \ref{lem:coefficient-influence}, followed by Parseval, give
$$
\frac{{\rm I}[f]^2}{4n}=
\sum_{i=1}^n\hat{f}(\{i\})^2\le
\sum_{S\ne\varnothing}\hat{f}(S)^2=
\alpha(1-\alpha).
$$
Thus
\begin{equation}\label{eq:influence-upper}
{\rm I}[f]\le 2\sqrt{n\alpha(1-\alpha)}.
\end{equation}
Combining \eqref{eq:influence-lower} and \eqref{eq:influence-upper}, we obtain $\alpha\le \frac1{1+n^{1/3}}$.
\end{proof}

\section{The entropy refinement}\label{sec:entropy}

\begin{proof}[Proof of Theorem~\ref{thm:main}]
Let $f=\mathbbm{1}_{\A}$ and set $\alpha=\E[f]=\frac{|\A|}{2^n}$. Let $X$ be uniformly distributed on $\{0,1\}^n$. Since $\A$ is regular, there is a number $\theta\in[0,1]$ such that, for every $i\in[n]$,
$$
\theta=\Pr(X_i=1\mid X\in\A).
$$
Equivalently, if $Z$ is chosen uniformly from $\A$, then
$\Pr(i\in Z)=\theta$ for every $i\in[n]$. Note that
\begin{equation}
    \begin{split}
   \hat{f}(\{i\})&=E\left[f(X)\chi_{\{i\}}(X)\right]=\E[f(X)(1-2X_i)]=\E[f(X)]-2\E[f(X)X_i]\\
   &=\alpha-2\Pr(X\in\A,\ X_i=1)=\alpha-2\Pr(X\in\A)\Pr(X_i=1\mid X\in\A)=\alpha(1-2\theta).
    \end{split}
\end{equation}
By Lemmas~\ref{lem:regular-influences} and \ref{lem:coefficient-influence},
$$
{\rm I}[f]=-2n\hat{f}(\{i\})=2n\alpha(2\theta-1).
$$
Combining this with~\eqref{eq:influence-lower}, we obtain
\begin{equation}\label{eq:marginal-bias}
2\theta-1\ge\frac{\alpha}{1-\alpha}.
\end{equation}

By Lemma~\ref{lem:entropy-regular},
$$
n\log2+\log\alpha=\log|\A|\le nh(\theta).
$$
Applying Lemma~\ref{lem:pinsker} and then \eqref{eq:marginal-bias}, we find that
$$
\log\frac1\alpha\ge2n\left(\theta-\frac12\right)^2\ge\frac n2\left(\frac{\alpha}{1-\alpha}\right)^2.
$$
This is exactly
$$
\log\frac{2^n}{|\A|}\ge \frac n2\left(\frac{|\A|}{2^n-|\A|}\right)^2.
$$

For the explicit bound, we weaken the preceding estimate to $\log(\frac{1}{\alpha})\ge \frac{n\alpha^2}{2}$ and set $y=n\alpha^2$. Since $\log(\frac{1}{\alpha})=\frac12\log\frac{n}{y}$, we have $\log\frac{n}{y}\ge y$, or equivalently $ye^y\le n$. The principal Lambert function is defined by $\W(x)e^{\W(x)}=x$ for $x\ge0$. For the Lambert function and its basic properties, see~\cite{C1996}. Its monotonicity therefore gives $y\le\W(n)$, and hence
$$
|\A|=2^n\alpha\le2^n\sqrt{\frac{\W(n)}n}.
$$
\end{proof}

\section{Concluding remarks}\label{sec:remark}

The estimates above are likely to be far from best possible. Ellis and Narayanan~\cite{EN2017} conjectured that every symmetric 3-wise intersecting family $\A\subseteq\mathcal P_n$ satisfies
$$
\log_2|\A|\le n-cn^\delta
$$
for some universal constants $c,\delta>0$, and Frankston, Kahn and Narayanan~\cite{FKN2018} asked for the same conclusion for regular increasing families. 

Let us indicate where the present argument loses information. With the notation used in the proofs, its two key conclusions are
$$
{\rm I}[f]\ge\frac{2n\alpha^2}{1-\alpha}
\qquad\text{and}\qquad
2\theta-1\ge
\frac{\alpha}{1-\alpha}.
$$
Subadditivity of entropy then gives
$$
\log\frac1\alpha\ge n\bigl(\log2-h(\theta)\bigr).
$$
Even if one keeps the exact entropy function, these inequalities yield only
$$
\log\frac1\alpha\ge n\left(
\log2-h\left(\frac1{2(1-\alpha)}\right)\right)=
\frac{n\alpha^2}{2}+O(n\alpha^3)
$$
as $\alpha\to0$. Thus the method naturally stops at the scale $\alpha\asymp\sqrt{\frac{\log n}{n}}$.

The main loss occurs when the cubic identity is estimated by
$$
\alpha^3\le\left(\max_{S\ne\varnothing}|\hat{f}(S)|\right)\sum_{S\ne\varnothing}\hat{f}(S)^2\le\frac{{\rm I}[f]}{2n}\alpha(1-\alpha).
$$
This step discards the signs of the Fourier coefficients and the way in which the Fourier mass is distributed across the different levels. Regularity and monotonicity make the coordinate influences equal, but it gives no higher-level information in the form used here. The entropy argument then sees only the common one-coordinate marginal and charges its deviation from $1/2$ quadratically.

\noindent\textbf{Acknowledgments.}
After the main mathematical results of this note had been obtained, ChatGPT 5.6 pointed out that the principal branch of the Lambert $W$-function could be used to express the bound in Theorem~\ref{thm:main} in a more explicit and comparable form. We also used ChatGPT 5.6 to assist with polishing the exposition. All other mathematical content in this paper is due to the authors.

\bibliographystyle{abbrv}
\bibliography{reference}
\end{document}